\documentclass[11pt]{article}

\usepackage{amssymb,amsfonts,amsmath,amsthm,latexsym,epsf,tikz,url}
\usepackage[usenames,dvipsnames]{pstricks}
\usepackage{pstricks-add}
\usepackage{epsfig}
\usepackage{pst-grad} 
\usepackage{pst-plot} 
\usepackage[space]{grffile} 
\usepackage[normalem]{ulem} 
\usepackage{etoolbox} 
\usepackage{float}
\usepackage{soul}
\usepackage{tikz}
\usepackage[colorinlistoftodos]{todonotes}
\usepackage{pgfplots}
\usepackage{mathrsfs}
\usepackage[colorlinks]{hyperref}
\usetikzlibrary{arrows}
\usepackage{graphicx} 
\usepackage{booktabs} 
\usepackage{tkz-graph}
\usepackage{pdflscape}
\usepackage{bm}
\usepackage{xparse}
\usepackage{pgffor}

\GraphInit[vstyle = Shade]
\makeatletter 
\patchcmd\Gread@eps{\@inputcheck#1 }{\@inputcheck"#1"\relax}{}{}
\makeatother

\newtheorem{theorem}{Theorem}[section]

\newtheorem{proposition}[theorem]{Proposition}

\newtheorem{corollary}[theorem]{Corollary}
\newtheorem{lemma}[theorem]{Lemma}

\newtheorem{definition}[theorem]{Definition}

\newcommand{\decisionpb}[4]{
	\begin{center}
		\noindent\framebox{\begin{minipage}{#4\textwidth}
				#1\\
				\textbf{Instance:} #2\\ 
				\textbf{Question:} #3
		\end{minipage}}
	\end{center}
}

\newcommand{\cdiez}[1]{%
	\begin{tikzpicture}
 		\draw[black, line width=1pt] (0,0) circle (1cm);

		\foreach \a [count=\i from 0] in {#1} {%
 			\pgfmathsetmacro{\angle}{\i*36}

	 \def\fillcolor{black}%
 			\if\a b\def\fillcolor{white}\fi
 			\if\a r\def\fillcolor{red}\fi
 			\if\a n\def\fillcolor{black}\fi
 \node[draw=black, fill=\fillcolor, circle, inner sep=0pt, minimum size=7pt] at (\angle:1cm) {};
 }
\end{tikzpicture}
}

\newcommand{\cquince}[1]{%
	\begin{tikzpicture}
 		\draw[black, line width=1pt] (0,0) circle (1.8cm);

		\foreach \a [count=\i from 0] in {#1} {%
 			\pgfmathsetmacro{\angle}{\i*360/15}

	 \def\fillcolor{black}%
 			\if\a b\def\fillcolor{white}\fi
 			\if\a r\def\fillcolor{red}\fi
			\if\a z\def\fillcolor{blue}\fi
 			\if\a n\def\fillcolor{black}\fi
 \node[draw=black, fill=\fillcolor, circle, inner sep=0pt, minimum size=7pt] at (\angle:1.8cm) {};
 }
\end{tikzpicture}
}

\begin{document}

\def\nt{\noindent}

\title{Locating-dominating coalitions in graphs}

\bigskip 
\author{ M. Chellali$^{1}$, A.A. Dobrynin$^2$, F. Foucaud$^{3}$,  \\[.5em]
 H. Golmohammadi$^{2}$, J.C. Valenzuela-Tripodoro$^{4}$
}


\maketitle
\begin{center}

$^{1}$LAMDA-RO Laboratory, Dept. of Mathematics, Univ. 
of Blida, Blida, Algeria.
\medskip

$^{2}$Sobolev Institute of Mathematics, Ak. Koptyug av. 4, Novosibirsk,
630090, Russia 
\medskip

$^{3}$Université Clermont Auvergne, CNRS, Clermont Auvergne INP,\\  
Mines Saint-Étienne, LIMOS, 63000 Clermont-Ferrand, France.

\medskip
$^{4}$Department of Mathematics, University of C\'{a}diz, Spain.
\medskip

\medskip
    	\tt m\_chellali@yahoo.com ~ dobr@math.nsc.ru ~ florent.foucaud@uca.fr  \\ 
        h.golmohammadi@g.nsu.ru  ~ jcarlos.valenzuela@uca.es
\end{center}

\begin{abstract}
A set $D$ of vertices in a graph $G = (V, E)$ is a locating-dominating set
(LD-set) if it is dominating and every two vertices $u$, $v$ of $V\setminus D$ satisfy 
$N(u) \cap D \ne N(v) \cap D$. Two disjoint sets $A,B\subset V(G)$ form a 
locating-dominating coalition (for short, an LD-coalition) in $G$ if none of them is 
an LD-set in $G$ but their union $A\cup B$ is an LD-set. A locating-dominating coalition 
partition (for short, an LDC-partition) is a vertex partition $\Pi$ such that 
every set of $\Pi$ is not an LD-set in $G,$ but forms an LD-coalition with 
another set of $\Pi$.
The locating-domination coalition number of $G$, denoted by $C_{L}(G),$ equals the 
maximum cardinality of an LDC-partition of $G$. Our purpose in this paper is to initiate 
the study of locating-dominating coalitions in graphs. We first investigate 
the existence of LDC-partitions. We also obtain lower and 
upper bounds on $C_{L}(G)$. We characterize connected graphs $G$ of order $n\ge 3$ satisfying 
$C_L(G) = n,$ as well as those trees $T$ such that $C_L(T)=n-1$. 
In addition, we determine the exact values of $C_L(G)$ for some classes of graphs. 
Moreover, we investigate the computational complexity of the decision problem associated 
with locating-dominating coalition partitions. To the best of our knowledge, this 
is the first work that addresses the algorithmic complexity of a decision problem 
related to coalition partitions, not only for this locating-dominating model but 
for coalition partitions in general.
\end{abstract}

  
\medskip
\noindent{\bf AMS Subj.\ Class.:}  05C60. 


\section{Introduction} \label{intro}

All the graphs considered in this paper are finite, connected and simple graphs. For graph 
theory notation and terminology, we generally follow \cite{W}. For such a graph $G=(V,E)$ and 
a vertex $v\in V$, we denote by $N(v)=\{w\in V \mid vw\in E\}$ the open neighborhood of $v$ and 
by $N[v]= N(v)\cup\{v\}$ its closed neighborhood. The 2-neighborhood of $v$ is defined as 
$N_2(v)=\{w\in V\mid 1 \leq d(v,w)\leq 2 \}$. The closed neighborhood for a 
vertex set A is $N[A]=\cup_{v\in A} N[v].$
Two distinct vertices $u$ and $v$ of a graph $G$ are twins if $N(u) = N(v)$. The order of a 
graph $G$ refers to the cardinality of its set of vertices. Each vertex of $N(v)$ is called a neighbor
of $v$, and the cardinality of $N(v)$ is called the degree of $v$, denoted by $\deg(v)$. The 
minimum and maximum degrees of the vertices of a graph are denoted 
by $\delta(G)$ and $\Delta(G)$, respectively. For an integer $k \geq 1$, we let 
$[k]$ denote the set $\{1,2, \dots, k\}$.

The idea of a coalition in graphs stems from domination in graphs and was initially 
introduced by Haynes et al. \cite{A11}. Following this foundational research, a considerable amount of 
literature has since been produced on the subject \cite{A4,A12,A13,A14}. Moreover, several variations 
of coalition have emerged, arising from domination-type concepts. For instance, see the papers 
\cite{A1,A2,A3,A5, bsb, A17}. In this paper, we study the locating-dominating coalitions 
relying on locating-domination.

\medskip

A set $S \subseteq V$ is called a \emph{dominating set} if every
vertex in $V \setminus S$ has a neighbor in $S$. To get some insights
into the vast research area of domination and its variations in
graphs, we refer the reader to the books ~\cite{A15,A16}. A
dominating set $D$ of $G$ is a \emph{locating-dominating set},
abbreviated LD-set, of $G$ if all vertices not in $D$ have pairwise
distinct open neighborhoods in $D$. 
The \emph{locating-domination number} of $G$, denoted by $\gamma_{L}(G)$, 
is the minimum cardinality among all LD-sets of $G$. The concept of a locating–dominating 
set was introduced and first studied by Slater in the 1980s~\cite{slater87, slater88} and 
has been studied extensively in \cite{A6,A7,A8,A9,A10}. A book chapter was dedicated to 
this topic~\cite{A15b} and an extensive online bibliography is maintained 
in~\cite{JL}.

\medskip

A great deal of research has been conducted on 
partitioning the vertex set of a graph $G$ into subsets with a specified 
property. A \emph{domatic partition} is a partition of $V(G)$ into dominating sets 
and the domatic number of $G$, denoted $d(G)$, equals the maximum size of 
a domatic partition of $G$. \cite[Chapter 12]{A15}. A \emph{locating-domatic partition}
is a partition of $V(G)$ into LD-sets. The \emph{location-domatic number}
of $G$, denoted $d_{loc}(G)$, equals the maximum size of a locating-domatic partition of $G$. 
In \cite{A18}, Zelinka studied this concept. A \emph{coalition} in a graph $G$ is 
composed of two disjoint sets of vertices $X$ and $Y$ of $G$, neither of which is a 
dominating set but whose union $X \cup Y$ is a dominating set of $G $. A \emph{coalition partition} 
is a vertex partition $\Pi=\{V_1,V_2,\dots,V_k\}$ of $V$ such that for every $i\in\{1,2,\dots,k\}$ 
the set $V_i$ is either a dominating set and $|V_i|=1$, or there exists another 
set $V_j$ so that they form a coalition. The maximum cardinality of a coalition 
partition is called the \emph{coalition number} of a graph, and is denoted by $C(G)$. 
{ The locating coalition graph with respect to the graph $G$, with vertex partition 
$\Pi=\{V_1,V_2,\dots,V_k\}$ is the graph with vertex set $\{V_j\}$ and 
where two vertices $V_i$ and $V_j$ are adjacent if and only if the sets $V_i$ and $V_j$ 
form an LD-coalition in $G$.}

\medskip

In this paper, we combine the notions of coalitions and locating-dominating sets. We now state the following definitions of the core concepts of this paper.

\begin{definition}
[Locating-dominating coalition] Two disjoint sets $X,Y\subseteq V(G)$ form a 
locating-dominating coalition (LD-coalition) in a graph $G$ if they are not 
LD-sets, but their union is an LD-set in $G$.
\end{definition}

\begin{definition}
[Locating-dominating coalition partition]\label{2.2} A 
locating-dominating coalition partition (LDC-partition),
in a graph $G$ is a vertex partition $\Pi
=\{V_{1},V_{2},\dots,V_{k}\}$ such that every $V_{i}$ is not an LD-set in $G,$
but forms an LD-coalition with another set $V_{j}$ for some $j$, where
$j\in [k]\setminus{\{i\}}$. The locating-domination coalition number, 
$C_{L}(G)$, in a graph $G$ equals the maximum size $k$ of an LDC-partition of $G$.
\end{definition}

For some examples, consider the graph $K_2$. It has no LDC-partition because if it 
had one, it should consist of two sets of size $1$, but since any single vertex is an LD-set
of $K_2$, it fails the definition. However, $K_3$ has an LDC-partition of 
size $3$ (each set is a single vertex), indeed, any two vertices form an LD-set 
of $K_3$, but a single vertex 
does not. Interestingly, $K_3$ has no LDC-partition of size $2$ because if it had one, 
it would consist of one set of size $2$ and one set of size $1$, but the part of size $2$ forms 
an LD-set of $K_3$. Consider now the complete graph $K_4$. Since any subset 
of vertices with cardinality at most $2$ is not an LD-set, it follows that $C_L(K_4) < 4$. However, as any 
three vertices form an LD-set, it is possible to find LDC-partitions of size $3$ consisting of one 2-element 
subset and two additional singleton sets.

\medskip

The structure of the paper is the following. In Section \ref{intro}, we establish notation and provide 
the main definitions used in the paper. In Section \ref{existence-bounds}, we discuss the existence of 
LDC-partitions in graphs and present some bounds on the locating-domination
coalition number, $C_L(G),$ of a graph $G$. In Section \ref{path-cycles}, we determine the 
locating-domination coalition number of paths and cycles. In Section \ref{n-1,n},  {we give a 
characterization of connected graphs $G$ of order $n\ge 3$ such that $C_L(G)=n$ and, besides, we describe 
those trees $T$ such that $C_L(T)=n-1$}. Finally, in Section~\ref{computational}, we 
study the decision problem associated with the LDC-partition problem.

\section{Existence and bounds}\label{existence-bounds}

In this section, we investigate the existence of  {an LDC-partition in graphs}, and we 
also establish some bounds on the  {locating-domination coalition number} of a graph $G$. 
Note that the graphs $K_1$ and $K_2$ do not admit any  {LDC-partition}. We first recall 
the following results.

\begin{theorem}[Slater \cite{slater88}]\label{th:logn}
For every graph $G$ of order $n$, if there is a locating-dominating set of size $k$, we have 
$n\leq 2^{k}+k-1$, and thus, $\gamma_L(G)\geq \lceil\log_2(n+1)-1\rceil$.
\end{theorem}

\begin{theorem}[Slater \cite{slater88}]\label{th:Slater_upb} 
For every connected graph of order $n\ge 2,$ $\gamma_{L}(G)\leq n-1,$ 
with equality if and only if $G$ is the star $K_{1,n-1}$ or the complete 
graph $K_{n}$.
\end{theorem}

 {\subsection{Existence and lower bounds.}}

The first result we state shows that every connected graph 
with diameter at least~3 admits an  {LDC-partition.}

\begin{proposition}\label{prop:diam3} 
Every connected graph $G$ of diameter at least 3 admits an  {LDC-partition.}
\end{proposition}

\begin{proof}
Consider a diametral path $P$ in $G,$ and let $x$ 
and $y$ be the endvertices of $P$. Consider a partition of $V(G)$ into two 
non-empty sets $X$ and $Y$ such that $X=N[x]$ and  {$Y=V(G)\setminus X$}. Since $G$
has diameter at least 3, $x$ has no neighbor in $Y$ and $y$ has no neighbor in $X$.    
Hence, neither $X$ nor $Y$ is an LD-set in $G$. However, both together form 
a trivial LD-set of $G$ and $\{X,Y\}$ is  {an LDC-partition} of $G$.
\end{proof}

\begin{theorem}~\label{th:domatic}
Let $G$ be a connected graph satisfying one of the following conditions,
\begin{enumerate}
    \item $G$ has diameter at least 3,
    \item The  {location-domatic number}, $d_{loc}(G),$ is at least 2. 
\end{enumerate}
Then $C_{L}(G)\ge 2 d_{loc}(G)$.
\end{theorem}

\begin{proof}
Let $\Pi=\{L_{1},\ldots,L_{k}\}$ be a locating-domatic partition of $G$ with $k=d_{loc}(G)$. 
If $k=1$, then by the assumptions we have that $diam(G)\ge 3$ and therefore, the result is a 
consequence of Proposition~\ref{prop:diam3}. So, let us assume that $k\ge 2.$

Without loss of generality, we may consider that $L_{1},...,L_{k}$ are 
minimal LD-sets. If $L_{i}$ is not a minimal LD-set of $G$ for some 
$i\in\lbrack k-1]$, then we can replace the set $L_{i}$ in $\Pi$ with 
a proper subset $L_{i}^{\prime}$ of $L_{i}$ that is a
minimal LD-set of $G$, and replace the set $L_{k}$ with the set $L_{k}%
\cup(L_{i}\setminus L_{i}^{\prime})$. Note that because of the diameter, $G$ has 
 {at least} 
four vertices, and thus every minimal LD-set of $G$ has cardinality at least~2. 
 {Hence, any partition of a minimal LD-set into two non-empty sets creates 
a pair of non-locating-dominating sets that together form an LD-coalition.}
Therefore, for every $i\in\lbrack k-1]$, we can split every non-singleton 
set $L_{i}$ into two sets
$L_{i,1}$, and $L_{i,2}$, which form  {an LD-coalition.} Doing this for each
$L_{i}$, for $i\in\lbrack k-1]$, we get a collection $\Pi^{\prime}$ of sets,
each of which is not an LD-set but forms  {an LD-coalition} with another set
in $\Pi^{\prime}$.

We now examine the set $L_{k}.$ We first note that since every superset of an
LD-set remains an LD-set, and thus the (new) set $L_{k}$ resulting from the
possible transfer of vertices from any set $L_{i}$ with $i<k$ is an LD-set in
$G.$ Now, if $L_{k}$ is minimal, then we can divide it into two 
non-locating-dominating sets, incorporate these two sets to $\Pi^{\prime}$ and 
construct an  {LDC-partition} of $G$ of size at least $2k$. Next suppose that 
$L_{k}$ is not a minimal LD-set. In this case, let $L_{k}^{\prime}\subseteq L_{k}$ 
be a minimal LD-set contained in $L_{k}$. We can divide $L_{k}^{\prime}=
L_{k,1}^{\prime }\cup L_{k,2}^{\prime}$ into two non-empty, 
non-locating-dominating sets, which together create a  {LD-coalition}, 
and let $L_{k}^{\prime\prime}=L_{k}\setminus L_{k}^{\prime}$ and add 
$L_{k,1}^{\prime}$ and $L_{k,2}^{\prime}$ to $\Pi^{\prime}$. Consequently 
$L_{k}^{\prime\prime}$ is not an LD-set, otherwise there are at least $k+1$ 
disjoint LD-sets in $G$, contradicting $k=d_{loc}(G)$. If $L_{k}^{\prime\prime}$ 
forms  {an LD-coalition} with any set of $\Pi^{\prime},$ then incorporating 
$L_{k}^{\prime\prime}$ to $\Pi^{\prime}$, we obtain an  {LDC-partition} of $G$ 
of size at least $2k+1$. Moreover, if $L_{k}^{\prime\prime}$ does not form 
 {an LD-coalition} with any set in $\Pi^{\prime}$, we eliminate $L_{k,2}^{\prime}$ 
from $\Pi^{\prime}$ and add the set $L_{k,2}^{\prime}\cup L_{k}^{\prime\prime}$, 
to $\Pi^{\prime}$. Hence, in either case, we have an  {LDC-partition} of $G$ of 
size at least $2k$, and this completes the proof. 

\end{proof}

The only graphs for which Theorem \ref{th:domatic} does not apply are those with diameter $2$ 
and  {location-domatic number} equals $1$. According to Theorem \ref{th:domatic}, 
the following question naturally arises: \emph{does a graph of diameter at most 2 admit an 
 {LDC-partition}?}

\medskip

Next, we provide some partial answers to the previous question.

\begin{proposition} \label{prop:ceil n/2} 
Let $G$ be a  {non trivial} graph with $\gamma_L(G)> \lceil\frac{n}{2}\rceil.$ 
Then $G$ admits an  {LDC-partition}.
\end{proposition}

\begin{proof} If $n=3,$ then $C_{L}(G)=3,$ because it suffices to consider
each vertex of $G$ as a singleton set. Hence, let $n\geq4.$ Each set of 
vertices of $G$ of cardinality at most $\lceil\frac{n}{2}\rceil$ is
not an LD-set. Consider a partition of $V(G)$ into two non-empty
sets $X$ and $Y$ such that  $\left\vert X\right\vert= \lfloor
\frac{n}{2}\rfloor$ and  $\left\vert Y\right\vert=\lceil
\frac{n}{2}\rceil$.  Clearly, $X\cup Y$ 
is a trivial LD-set of $G,$ leading us to the desired result. 
\end{proof}

As an immediate consequence of Theorem~\ref{th:Slater_upb}, 
Theorem~\ref{th:domatic} and Proposition~\ref{prop:ceil n/2},
we obtain the following. 

\begin{corollary}\label{coro}  Let $G$ be a graph of order $n$.
If $G\in\{K_{1,n-1},K_{n}\}$ or $G$ is a graph
with $diam(G)\ge 3$ or $d_{loc}(G)\ge 2$, then $C_{L}(G)\ge 2.$ 
\end{corollary}

 The following result applies to all graphs of diameter $2$, not just stars. 
 {So, regarding graphs having diameter $2$, it generalizes Corollary \ref{coro}.}

\begin{proposition}\label{prop:tweens}
    Let $G$ be a graph of order $n\geq 4$ containing twins. Then, $G$ 
    admits an  {LDC-partition} and $C_{L}(G)\ge 2.$
\end{proposition}

\begin{proof}
Let $u,v$ be two twin vertices in $G$ and let $X=\{u,v\}, Y=V(G)\setminus X.$
Since $n\ge 4$, then neither $X$ nor $Y$ are LD-sets. However,
$X\cup Y=V(G)$ is a trivial LD-set. Therefore, $\{X,Y\}$ is an LDC-partition and
$C_L(G)\ge 2.$

\end{proof}

\subsection{Upper bounds}

Trivially, for every connected graph $G$ of order $n\ge 3,$ $C_{L}(G)\leq n.$
Our next result provides a sharper upper bound on the locating-domination coalition 
number in terms of the order and the locating-domination number.

\begin{proposition} \label{prop0} Let $G$ be a connected graph of order $n\ge 3.$ Then
$C_{L}(G)\leq n-\gamma_{L}(G)+2\leq n- \lceil\log_2(n+1)-1\rceil+2$.
\end{proposition}

\textbf{Proof. }Let $C_{L}(G)=k,$ and consider  {an LDC-partition
$\Pi=\{V_{1},V_{2},...,V_{k}\}$ of maximum cardinality.} By definition, each $V_{i}$ forms 
 {an LD-coalition} in $G$ with some set $V_{j}$ such that $i\neq j.$ Without
loss of generality, assume that $V_{1}\cup V_{2}$ is  {an LD-coalition}. It
follows that $V_{1}\cup V_{2}$ is  {an LD-set in} $G,$ and thus
$\gamma_{L}(G)\le \left\vert V_{1}\right\vert +\left\vert V_{2}\right\vert .$ 
Accordingly, since $|V_j|\ge 1$ we have that
\[
n   =\sum_{j=1}^k |V_j| \ge |V_{1}| +|V_2| +(k-2) \ge \gamma_{L}(G)+(k-2),
\]

leading to the desired upper bound. The second inequality follows from Theorem~\ref{th:logn}. $\Box$\newline

The following corollary is an immediate consequence of Proposition
\ref{prop0}. Recall that by Theorem~\ref{th:logn}, for every graph $G$ of order~$n$,
we have $n\leq \gamma_L(G)+2^{\gamma_L(G)}-1$ and so, if $n\geq3$,
$\gamma_{L}(G)\geq 2$.

\begin{corollary}
\label{C_L=n}If $G$ is a connected graph of order $n\geq3$ such that
$C_{L}(G)=n,$ then $\gamma_{L}(G)=2$ and $n\leq 5$.
\end{corollary}

Note that the converse of Corollary \ref{C_L=n} is not true, as can be seen by
the path $P_{5},$ where $\gamma_{L}(P_{5})=2,$ but $C_{L}(P_{5})<5$, since the
center vertex of the path cannot form an LD-coalition with any other
vertex of $P_5.$

\begin{figure}[!t]
	\begin{center}
		\includegraphics[width=0.47\linewidth]{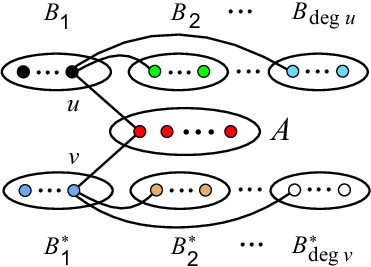}
		\caption{ The maximal number of partners for the set $A$
		}
		\label{FigBound}
	\end{center}
\end{figure}

\begin{lemma} \label{t}
Let $G$ be a connected graph of order $n\geq3$ and $\Pi$ be an  {$LDC$-partition}. 
For any $A\in \Pi$, the number of  {LD-coalition partners} of $A$ is at most 
$2\Delta(G)$, and this bound is sharp.  
\end{lemma}

\begin{proof}
Let $\Pi=\{V_1,V_2,\ldots ,V_k\}$ be an LDC-partition 
and $A\in \Pi$. It follows that either $A$ is a non-dominating set or it is a dominating set, 
but is not a locating-set in $G$. 

First, let us suppose that $A$ is a non-dominating set. Note that every LDC-partition of a 
graph is always a dominating partition of the graph. Haynes et al. \cite{A11} proved that the 
number of coalition partners of any set in a dominating partition is at most $\Delta(G)+ 1$. Hence, 
the number of LD-coalition partners for $A$ also does not exceed $\Delta(G)+1$, and the result
holds.

Suppose now that $A$ is a dominating set but is not a locating-dominating set in $G$. Then 
there exist vertices $u,v \in  V(G)\setminus A$ such that $N[u] \cap A = N[v] \cap A \not = \emptyset$.
If a set $B \in \Pi$ forms an LD-coalition with $A$, then $A \cup B$ is an LD-set of $G$. 
Hence, the set $B$ must contain at least one member of the set $\left( N[u] \cup N[v]\right) \setminus A$. 
Therefore, there are at most $2\Delta(G)$ sets, such as $B,$ that may form an LD-coalition with $A$ 
(see Figure~\ref{FigBound}). 

To see the sharpness, consider the cubic graph depicted in Figure~\ref{FigBoundExample}. 
The set $A$ consisting of red vertices is a dominating set. Since vertices $u$ and $v$ have the 
same neighbor $w \in A$, $A$ is not an LD-set. It is not hard to check that the union of $A$ and 
each one of the remaining vertices forms an LD-set. 
\end{proof}

\begin{figure}[!t]
	\begin{center}
		\includegraphics[width=0.8\linewidth]{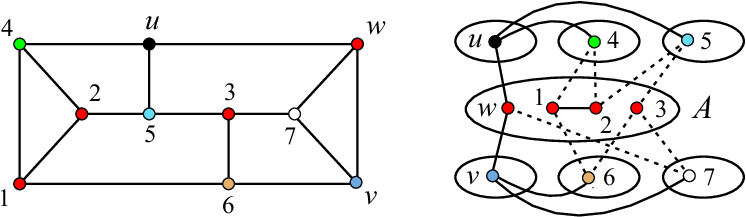}
		\caption{Cubic graph with the set $A$ having $2\Delta$ coalition partners 
		}
		\label{FigBoundExample}
	\end{center}
\end{figure}

\section {Paths and cycles} \label{path-cycles}

In this section, we obtain the exact value of the  {locating-domination coalition number
of paths and cycles. We first recall the following results.}

\begin{lemma} [\cite{A11}] \label{l1}
For any $n\ge 3$, $C(C_n)\le 6$. 
\end{lemma}

\begin{lemma} [\cite{A11}] \label{l2}
For the path $P_n$, $C(P_n) \le 6$. 
\end{lemma}

\begin{theorem} [\cite{slater88}]  \label{t1}
For $n\ge7$ , $\gamma_{L}(C_n)=\gamma_{L}(P_n)=\lceil\frac{2n}{5}\rceil$. 
\end{theorem}

To determine the exact value of the locating-domination coalition number of a cycle, 
we first prove that this value is at most 6 for any cycle.

\begin{lemma} \label{TCL6}
For the cycle $C_n$ of order $n\ge 3$, $C_{L}(C_n) \le 6$. 
\end{lemma}

\begin{proof}
Since $\gamma_L(C_n)=\lceil \frac{2n}{5} \rceil$ and we know that 
$C_L(G)\le n-\gamma_L(G)+2$, we may derive that $C_L(C_n)\le 6$ for $3 \le n \le 8.$ Hence, from 
now on, let $n\ge 9$ be an integer.

\medskip

Let ${\cal P}_{L}(C_n)$ and ${\cal P}(C_n)$ be the sets of all LDC-partitions and dominating 
coalition partitions of $C_n$, respectively. If $\Pi=\{V_1, V_2,\dots,V_m\} \in {\cal P}_{L}(C_n)$, 
then by definition no set $V_i$ is an LD-set. This implies that for each $V_i$, either $V_i$ is not 
a dominating set, or $V_i$ is a dominating set but fails to be a locating set (that is, there exist 
distinct vertices $u, v \notin V_i$ such that $V_i \cap N(u) = V_i \cap N(v)$).

If every set in $\Pi$ is a non-dominating set, then $\Pi \in {\cal P}(C_n)$ since any union forming an 
LD-coalition is necessarily a dominating set. Therefore, $|\Pi|\le 6$ by Lemma~\ref{l1}.

\medskip

Now, assume $\Pi \in {\cal P}_{L}(C_n) \setminus {\cal P}(C_n)$. Then $\Pi$ must contain at least one 
dominating set. Let $V_1 \in \Pi$ be a dominating set which is not an LD-set. Thus, there exist 
vertices $u, v \notin V_1$ such that $V_1 \cap N(u) = V_1 \cap N(v)$. Since $V_1$ is dominating, this 
intersection cannot be empty; let $V_1 \cap N(u) = V_1 \cap N(v) = \{w\}$. Let us denote by $u_1$ and 
$v_1$ the neighbors of $u$ and $v$, respectively, distinct from $w$. Thus, we have $N(u)=\{w, u_1\}$ and 
$N(v)=\{w, v_1\}.$

Since $n \ge 9$ and $V_1 \cap N(u) = V_1 \cap N(v)$, the vertices $u_1,v_1$ are different and do
not belong to $V_1$. Otherwise, a cycle of length $4$ would be closed in $C_n$.

Suppose that $V_i$ and $V_1$ form an LD-coalition for some $i > 1$. To satisfy the locating condition 
$(V_1 \cup V_i) \cap N(u) \neq (V_1 \cup V_i) \cap N(v)$, $V_i$ must contain at least one vertex 
from $\{u, v, u_1, v_1\}$. So, we can have, at most, four subsets $V_{u_1},V_u,V_v,V_{v_1}$ in $\Pi$ 
that can form an LD-coalition with $V_1$.

The partition $\Pi$ is formed by $V_1$, the possible $V_1$-partners to form an LD-coalition (at most $4$ 
sets), and, eventually, some other subsets that cannot form an LD-coalition with $V_1$.

Consider any coalition $V_a \cup V_b$ that does not include $V_1$. Since any LD-coalition must be a dominating 
set, $V_a \cup V_b$ must dominate vertex $w$. However, $w \in V_1$, so $w \not\in V_a \cup V_b$. Therefore, 
$V_a \cup V_b$ must contain at least one neighbor of $w$. The only neighbors of $w$ are $u$ and $v$. Since 
$\Pi$ is a partition, $u$ belongs to exactly one set $V_u$, and $v$ belongs to exactly one set $V_v$ (not 
necessarily different from $V_u$). Consequently, any coalition disjoint from $V_1$ must include $V_u$ or $V_v.$ 

\begin{figure}[!ht]
\centering
\includegraphics[width=0.3\linewidth]{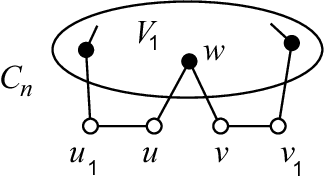}
\caption{Vertices $u$ and $v$  of $C_n$ with  $V_1 \cap  N(u) = V_1  \cap N(v)$.}  
\label{FigCL6}
\end{figure}

Moreover, let us assume that $V_a=V_u$. Since $V_b\neq V_1$, $N(v)=\{w, v_1\}$ and $v$ must be dominated by 
$V_u\cup V_b$, we may deduce that $V_b\in\{V_v,V_{v_1}\}.$
Summing up, in this case we have that $|\Pi | \le 5,$ and the result holds.

\end{proof}

To establish the exact locating-domination coalition number of cycles, we first prove several technical lemmas.
Let us denote by $V(C_{15})=\{v_j:j\in [15]\}$; $A=\{w_i\}\subset V(C_{15})$ any subset of vertices; and 
let us define the ``gap'' between vertices $w_j$ and $w_{j+1},$ denoted by $g_j=d(w_j,w_{j+1})-1.$ We 
label the vertices counterclockwise in the cycle. In Figure~\ref{ex_gap}, the blue vertex is $v_1.$

We define the vector $g(A) = [g_1, \ldots, g_{|A|} ]$ to be the gap configuration of $A$.
Clearly, the gap configuration determines, up to rotation, a subset of vertices in the cycle. If we 
fix $w_1=v_1$, then we have a unique configuration of vertices for each gap configuration 
(see Figure \ref{ex_gap}).

\begin{figure}[!ht]
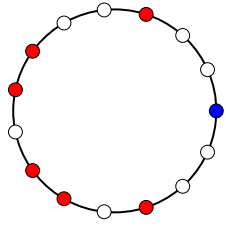

\centering
\scalebox{0.75}{
\begin{tabular}{c}
	\cquince{z,b,b,r,b,b,r,r,b,r,r,b,r,b,b}  
\end{tabular}
}
\caption{The blue vertex is $w_1=v_1$; $w_2$ is above and to the left of $w_1$; $A=\{ \text{non-white
vertices} \}$ has gap configuration $[2,2,0,1,0,1,2]$.} \label{ex_gap}
\end{figure}

\bigskip

First, we prove that a 5-set of vertices in $C_{15}$ which is not an LD-set can only form an 
LD-coalition with exactly one singleton set.

\begin{lemma}\label{LD5-1}
Let $A = \{w_j:j\in[5]\} \subset V(C_{15})$ be a non LD-set. Let $w_6 \in V(C_{15}) \setminus A$ be such that 
$A \cup \{w_6\}$ is an LD-set.  Then, for every $x \in V(C_{15}) \setminus  \{w_6\}$, the set 
$A \cup \{x\}$ is not an LD-set.
\end{lemma}

\begin{proof} 
Let $A' = A \cup \{w_6\}$ and let $[g_1,g_2,\ldots,g_6]$ denote the gap configuration of $A'$. 
Observe that $g_j < 3$, since otherwise $A'$ would fail to be a dominating set of $C_{15}$. 
Moreover, if some gap, say $g_1$, is equal to $0$, then $|N[\{w_1,w_6\}]| = 4$, and the remaining four 
vertices can dominate at most $12$ vertices. 
Therefore, $A'$ fails to dominate $V(C_{15})$, and consequently it cannot be a locating-dominating 
set of $C_{15}$ (see Figure~\ref{ex_gap}).

So, $g_j\in\{1,2\},$ for all $j\in [6].$

Let $p$ denote the number of $1$-gaps and let $q$ be the number of $2$-gaps. Clearly, 
\[p+q = 6 \quad \text{and} \quad p + 2q = |V(C_{15}) \setminus A'| = 9.\] 
From this, one can easily derive that $p = q = 3$. Finally, no two consecutive gaps can both have value $2$, 
because the neighbors of the single vertex between them would not be locating-dominated by $A'$. 
See Figure~\ref{ex_gap}, the neighbors {of $w_1$} cannot be locating-dominated by $A'$.

Hence, up to rotation, the set $A'$ admits a unique gap configuration, \( [2,1,2,1,2,1], \)
which implies that, again up to rotation, $A'$ has a unique vertex configuration (see Figure~\ref{C15-1}~(a)).

\begin{figure}[!ht]
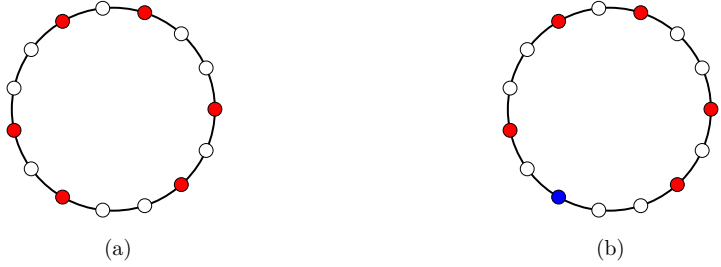

\centering
\scalebox{0.75}{
\begin{tabular}{cp{4cm}c}
	\cquince{r,b,b,r,b,r,b,b,r,b,r,b,b,r,b}  &  & \cquince{r,b,b,r,b,r,b,b,r,b,z,b,b,r,b}  \\[.5em]
  		(a) &   &   (b) \\[1em]
\end{tabular}
}
\caption{ (a) The set $A'$ (b) If $A=\{\mbox{red vertices}\}$, the singleton is unique (blue).}\label{C15-1}
\end{figure}

A singleton set partners A in an LD-coalition only if $A$ is a subset of the unique configuration of
the LD-set $A'$ {(see Figure~\ref{C15-1}~(b))}, with the blue vertex representing the singleton that completes A into the LD-set {$A'$}.
\end{proof}


Now, we prove that a 6-set of vertices in $C_{15}$ which is not an LD-set can only form an 
LD-coalition with at most three singleton sets.

\begin{lemma}\label{LD5-2}
Let $A = \{w_j:j\in[6]\} \subset V(C_{15})$ be a non LD-set. Then there exist at most three vertices 
$w\in V(C_{15}) \setminus A$ such that $A \cup \{w\}$ is an LD-set. Moreover, if there are
three of such singleton sets then $|V\setminus N[A]|=1$ and $N[A]$ induces a connected subgraph.
\end{lemma}

\begin{proof} Let $A' = A \cup \{w\}$, and let $[g_1, g_2, \dots, g_6]$ denote the gap configuration of $A$. 
If $N[A]$ induces a disconnected subgraph in $C_{15}$, see Figure~\ref{C15-2}~(a), or $|V \setminus N[A]| \ge 4$, 
see Figure~\ref{C15-2}~(b), then no singleton can form an LD-set with $A$, as it would fail to dominate the 
cycle, yielding a contradiction. Hence, we assume that the induced subgraph by $N[A]$ is connected.

\begin{figure}[!hb]
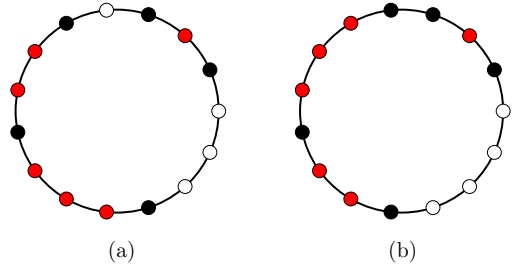

\centering
\scalebox{0.75}{
\begin{tabular}{cp{.25cm}c}
	\cquince{b,n,r,n,b,n,r,r,n,r,r,r,n,b,b}  &  & \cquince{b,n,r,n,n,r,r,r,n,r,r,n,b,b,b}  \\[.5em]
  		(a) &   &   (b) \\[1em]
\end{tabular}
}
\caption{If (a) $N[A]$ is disconnected or, (b) $|V\setminus N[A]|\ge 4$, then no singleton can dominate the white 
vertices.  }\label{C15-2}
\end{figure}

Moreover, if $|V \setminus N[A]| = 3$, then the only singleton that can form an LD-coalition with $A$ is the 
central vertex of $V \setminus N[A]$. If $|V \setminus N[A]| = 2$, then the only such singletons are precisely the 
two vertices of $V \setminus N[A]$. If $|V \setminus N[A]| = 1$, and we are looking for singleton sets that can 
dominate the only vertex $x$ in $V \setminus N[A],$ then the candidates are, at most, precisely the three vertices 
of $N(x)$ (see Figure~\ref{C15-3}). 

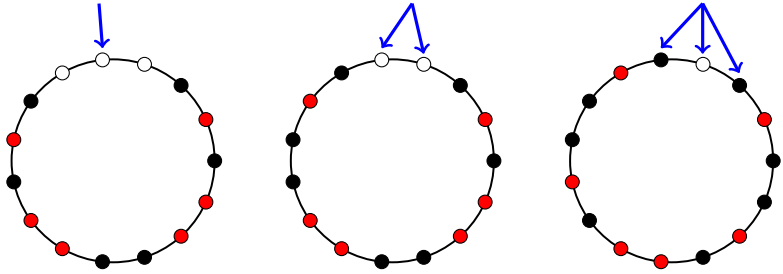
\begin{figure}[!h]
\centering
\begin{tabular}{ccccc}
\begin{tikzpicture}[scale=.75]
 	\draw[black, line width=.75pt] (0,0) circle (1.8cm);

	\foreach \a [count=\i from 0] in {n,r,n,b,b,b,n,r,n,r,r,n,n,r,r} {%
 		\pgfmathsetmacro{\angle}{\i*360/15}
		\def\fillcolor{black}%
 			\if\a b\def\fillcolor{white}\fi
 			\if\a r\def\fillcolor{red}\fi
			\if\a z\def\fillcolor{blue}\fi
 			\if\a n\def\fillcolor{black}\fi
 		\node[draw=black, fill=\fillcolor, circle, inner sep=0pt, minimum size=5.25pt] at (\angle:1.8cm) {};
 	}
	\draw[blue, very thick, ->] (-0.25,2.790) -- (-0.188,2);
	\end{tikzpicture}
& &
\begin{tikzpicture}[scale=.75]
 	\draw[black, line width=.75pt] (0,0) circle (1.8cm);

	\foreach \a [count=\i from 0] in {n,r,n,b,b,n,r,n,n,r,r,n,n,r,r} {%
 		\pgfmathsetmacro{\angle}{\i*360/15}
		\def\fillcolor{black}%
 			\if\a b\def\fillcolor{white}\fi
 			\if\a r\def\fillcolor{red}\fi
			\if\a z\def\fillcolor{blue}\fi
 			\if\a n\def\fillcolor{black}\fi
 		\node[draw=black, fill=\fillcolor, circle, inner sep=0pt, minimum size=5.25pt] at (\angle:1.8cm) {};
 	}
	\draw[blue, very thick, ->] (0.35,2.790) -- (-0.188,2); 
	\draw[blue, very thick, ->] (0.35,2.790) -- (0.556,1.9); 
\end{tikzpicture}
& &
\begin{tikzpicture}[scale=.75]
 	\draw[black, line width=.75pt] (0,0) circle (1.8cm);

	\foreach \a [count=\i from 0] in {n,r,n,b,n,r,n,n,r,n,r,r,n,r,n} {%
 		\pgfmathsetmacro{\angle}{\i*360/15}
		\def\fillcolor{black}%
 			\if\a b\def\fillcolor{white}\fi
 			\if\a r\def\fillcolor{red}\fi
			\if\a z\def\fillcolor{blue}\fi
 			\if\a n\def\fillcolor{black}\fi
 		\node[draw=black, fill=\fillcolor, circle, inner sep=0pt, minimum size=5.25pt] at (\angle:1.8cm) {};
 	}
	\draw[blue, very thick, ->] (0.55,2.790) -- (-0.188,2);
	\draw[blue, very thick, ->] (0.55,2.790) -- (1.2,1.559);
	\draw[blue, very thick, ->] (0.55,2.790) -- (.55,1.9);

\end{tikzpicture}
\end{tabular}
\caption{ The blue arrows show the only feasible singleton partners of $A$ to an LD-coalition. }\label{C15-3}
\end{figure}

Finally, assume that $|V \setminus N[A]| = 0,$ which implies that $A$ is a dominating set (although it is not 
an LD-set) in $C_{15}$. Since $A$ is a dominating set, reasoning as in the previous lemma, no gap $g_j$ is 
greater than or equal to $3$ in the gap configuration $[g_1,\ldots,g_6]$ of $A$. Hence, $g_j \in \{0,1,2\}$ for 
all $j \in [6]$. Let $p$, $q$, and $r$ denote the numbers of $0$-, $1$-, and $2$-gaps, respectively. 
Clearly, \[p + q + r = 6 \quad \text{and} \quad {q+ 2r = |V(C_{15}) \setminus A| = 9}.\]
One can readily deduce that there are only two possibilities: {$\{p=0,\,q=r=3\}$ or 
$\{p=q=1,\,r=4\}$.} Note that three 2-gaps cannot appear consecutively in the gap configuration, because 
this would create two pairs of vertices that share only one common neighbor each.  For example, in 
Figure~\ref{ex_gap}, no singleton can simultaneously help $A$ to locating-dominating the neighbors of $w_1$ 
and those of $w_4$.

Let us start with the case $\{p=0,\,q=r=3\}$. Since $A$ is not an LD-set, the gap 
configuration $[2,1,2,1,2,1]$ is not feasible and so, it must be either $[2,2,1,1,2,1]$ or $[2,2,1,2,1,1]$.
In both cases, no singleton vertex outside $N(w_2)$ can form an LD-coalition with $A$, 
since it would not distinguish the neighbors of $w_2$, which would continue to share identical neighborhoods 
in $A'$. So, at most two vertices can form an LD-coalition with $A$.

On the other hand, if $\{p=q=1,\,r=4\},$ the only possible gap configurations of $A$ with 
no three consecutive 2-gaps are $[2,2,1,2,2,0]$ and $[2,2,0,2,2,1]$, which are equivalent up 
to rotation. Hence, in the remaining case, $[2,2,0,2,2,1]$, there must be two 
pairs of consecutive 2-gaps. Consequently, no singleton can be added to $A$ to form an LD-coalition, as a 
single vertex cannot compensate for the failure of four distinct vertices to be 
simultaneously located and dominated by $A'$.

\end{proof}

\bigskip

\begin{proposition} \label{p1} For any cycle $C_n$,
$$
C_L(C_n) = \left\{
\begin{array}{rl}
n,  &  \mbox{if } n = 3,4,5; \\
5,  &  \mbox{if }   6 \le n \le 11 \mbox{ and \ }  n =13,15;\\
6,  &  \mbox{if }  n = 12, 14 \mbox{ and \ } n \ge 16. \\
\end{array}
\right.
$$
\end{proposition}

\begin{proof}
For small values of $n$, say $n \in \{ 3,4,5\}$, we can readily check that $C_L(C_n)=n$.
Let us point out that by Proposition~\ref{prop0} and Theorem~\ref{t1} we may deduce 
that $C_L(C_6)\le 5.$ 
Besides, in Figure~\ref{FigCL5} we describe an LDC-partition of cardinality $5$ for $C_n,$ whenever $n\ge 5.$ 
Therefore, $C_L(C_6)=5.$ Now let $n \ge 7.$ 

\begin{figure}[!ht]
\centering
\includegraphics[width=0.3\linewidth]{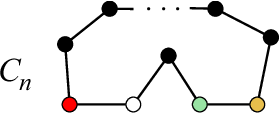}
\caption{An LDC-partition of cardinality 5 in cycle $C_n$ of order $n\ge 5$.}  
\label{FigCL5}
\end{figure}

We proceed further by showing that $C_L(C_n)\neq 6$, where $7\le n \le 11$ and $n=13,15$. 
In Table~\ref{TabPart}, all possible vertex partitions of the cycles $C_n$ are presented and labeled 
with labels in the set $\{1,2\}.$  The following criteria summarize the reasoning applied to the partitions:
\begin{enumerate}
\item[(label 1)] If for any $V_j \in \Pi$, $\max_{i \neq j} \{|V_i| + |V_j|\} < \gamma_L(C_n)$, the 
partition is impossible by {Theorem~\ref{t1}.}
\item[(label 2)] If a set must be the partner to more than $4$ sets, the partition is impossible by 
{Lemma~\ref{t}.}
\end{enumerate}

\begin{table}[t!] 
\small
\centering
\caption{All vertex partitions of cycle $C_n$.}  \label{TabPart}
\resizebox{\textwidth}{!}{%
\begin{tabular}{c c l l l l l }  \toprule
$n$ & $\gamma_{L}(C_n)$ &  \multicolumn{5}{c}{Partitions}                                           \\ \midrule 
 7  &    3              &  \multicolumn{5}{l}{(2,1,1,1,1,1)$^{1,2}$}                                \\ \midrule
 8  &    4              & (3,1,1,1,1,1)$^{1,2}$  &  (2,2,1,1,1,1)$^1$                               \\ \midrule  
 9  &    4              & (4,1,1,1,1,1)$^{1,2}$  &  (3,2,1,1,1,1)$^{1,2}$  &  (2,2,2,1,1,1)$^1$     \\ \midrule
10  &    4              & (5,1,1,1,1,1)$^{1,2}$  &  (4,2,1,1,1,1)$^{1,2}$  &  (3,2,2,1,1,1)  
                        & (2,2,2,2,1,1)$^1$                                                         \\ \midrule
11      &        5      & (6,1,1,1,1,1)$^{1,2}$  &  (5,2,1,1,1,1)$^{1,2}$  & (4,3,1,1,1,1)$^{1,2}$  
                        & (4,2,2,1,1,1)$^{1,2}$  &  (3,3,2,1,1,1)$^1$                               \\  
        &               & (3,2,2,2,1,1)$^1$      &  (2,2,2,2,2,1)$^1$      & \multicolumn{2}{c}{}   \\ \midrule
13      &        6      & (8,1,1,1,1,1)$^{1,2}$  &  (7,2,1,1,1,1)$^{1,2}$  & (6,3,1,1,1,1)$^{1,2}$   
                        & (6,2,2,1,1,1)$^{1,2}$  &  (5,4,1,1,1,1)$^{1,2}$                           \\ 
        &               & (5,3,2,1,1,1)$^{1,2}$  &  (5,2,2,2,1,1)$^{1,2}$  & (4,4,2,1,1,1)$^1$  
                        & (4,3,2,2,1,1)$^1$      &  (4,2,2,2,2,1)$^1$                               \\ 
        &               & (3,3,3,2,1,1)$^1$      &  (3,3,2,2,2,1)$^1$      & (3,2,2,2,2,2)$^1$ & &  \\ \midrule
15      &       6       & (10,1,1,1,1,1)$^{1,2}$ &  (9,2,1,1,1,1)$^{1,2}$  & (8,3,1,1,1,1)$^{1,2}$   
                        & (8,2,2,1,1,1)$^{1,2}$  &  (7,4,1,1,1,1)$^{1,2}$                           \\ 
        &               & (7,3,2,1,1,1)$^{1,2}$  &  (7,2,2,2,1,1)$^{1,2}$  & (6,5,1,1,1,1)         
                        & (6,4,2,1,1,1)          &  (6,3,2,2,1,1)$^{1,2}$                           \\ 
        &               & (6,2,2,2,2,1)$^{1,2}$  &  (5,5,2,1,1,1)          & (5,4,2,2,1,1) 
                        & (5,3,2,2,2,1)$^{1,2}$  &  (5,2,2,2,2,2)$^{1,2}$                           \\  
        &               & (4,4,4,1,1,1)$^1$      &  (4,4,3,2,1,1)$^1$      & (4,3,3,2,2,1)$^1$ 
                        & (4,3,2,2,2,2)$^{1,2}$  &  (3,3,3,3,2,1)$^1$                               \\  
        &               & (3,3,3,2,2,2)$^1$      & \multicolumn{4}{c}{}                             \\ 
\bottomrule
\end{tabular}
}
\end{table}

As a matter of example, the partition $(2,1,1,1,1,1)$ would not be possible because $\gamma_L(C_7)=3$, hence
any singleton set needs the 2-set to form an LD-coalition. This means that the 2-set has five different sets to form
an LD-coalition, which contradicts {Lemma~\ref{t}}. Therefore, this partition appears in Table \ref{TabPart} as 
$(2,1,1,1,1,1)^{(1,2)}.$

By applying the aforementioned, nearly all possible 6-set partitions of the cycles $C_n$ for $7 \le n \le 15$ can be 
ruled out. The remaining feasible configurations that warrant further investigation are the following: 
$(3,2,2,1,1,1)$ for the cycle $C_{10}$, and $(6,5,1,1,1,1)$, $(6,4,2,1,1,1)$, $(5,5,2,1,1,1)$, and $(5,4,2,2,1,1)$ 
for the cycle $C_{15}$. To conclude the proof, a detailed examination of these specific cases is provided below.

Consider an LDC-partition of $C_{10}$ of type $(3,2,2,1,1,1)$, denoted by $(A_1,A_2,\dots,A_6)$. 
Figure~\ref{C10} shows all possible configurations of the 3-set $A_1$ (red vertices), its neighbors (black), 
and vertices not dominated by $A_1$ (white). Since $\gamma_L(C_{10})=4$, each singleton $A_4,A_5,A_6$ must 
form an LD-coalition with $A_1$. If the subgraph of $C_{10}$ induced by $N[A_1]$ is disconnected (Figure~\ref{C10}(a)--(b)), 
then there exist non-dominated vertices $u,v$ with $d(u,v)\ge 4$, which no singleton can dominate, a contradiction. 
Thus, $N[A_1]$ induces a connected subgraph in $C_{10}$.

\begin{figure}[!ht]
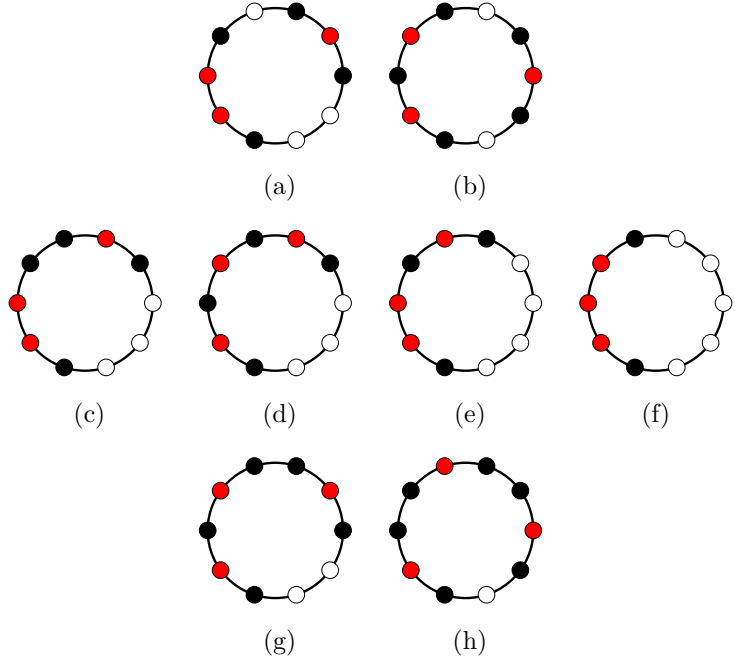

\centering
\scalebox{0.9}{
\begin{tabular}{cccc}
	& \cdiez{n,r,n,b,n,r,r,n,b,b} & \cdiez{r,n,b,n,r,n,r,n,b,n} & \\[.5em]
	&   (a) & (b) \\[1em]
\cdiez{b,n,r,n,n,r,r,n,b,b}  & \cdiez{b,n,r,n,r,n,r,n,b,b} & \cdiez{b,b,n,r,n,r,r,n,b,b} & \cdiez{b,b,b,n,r,r,r,n,b,b} \\[.5em]
  (c) & (d)  &  (e) & (f) \\[1em]
  	& \cdiez{n,r,n,n,r,n,r,n,b,b} & \cdiez{r,n,n,r,n,n,r,n,b,n} & \\[.5em]
	&   (g) & (h) 
\end{tabular}
}
\caption{$A=\{\mbox{red vertices}\}\in \Pi$ is the 3-set; $V\setminus N[A]=\{\mbox{white vertices}\}$ must be dominated by 
a singleton set. }\label{C10}
\end{figure}

If the set of non-dominated vertices by $A_1$ is such that $|V\setminus N[A_1]|\ge 4$ (Figure~\ref{C10} (e)--(f)), no 
singleton can help $A_1$ dominate them. 
If $|V\setminus N[A_1]|=3$ (Figure~\ref{C10}(c)--(d)), exactly one singleton (the central white vertex) can partner with $A_1$. 
If $|V\setminus N[A_1]|=2$ (Figure~\ref{C10}(g)), only those two white vertices can form LD-coalitions with $A_1$. 
In all cases, at most two singletons can partner with $A_1$, contradicting the need for three.

Otherwise, if $|V\setminus N[A_1]|=1$ (Figure~\ref{C10}(h)), the three singletons must be exactly the 
closed neighborhood of this white vertex to dominate it. However, the neighbors of the red vertex diametrically 
opposite the white vertex then share the same neighborhood intersection with any coalition of $A_1$ with a singleton, 
a contradiction.

Next, consider the configurations $(5,5,2,1,1,1)$ and $(5,4,2,2,1,1)$ in $C_{15}$. 
Since $\gamma_L(C_{15}) = 6$, the three singleton sets must form an LD-coalition with a 5-set. 
However, by Lemma~\ref{LD5-1}, a 5-set that is not an LD-set can form an LD-coalition 
with at most one singleton, which contradicts the requirement of three.

Finally, consider an LDC-partition $(A_1,A_2,\dots,A_6)$ of $C_{15}$. We show that 
the configurations $(6,5,1,1,1,1)$ and $(6,4,2,1,1,1)$ are infeasible by Lemmas~\ref{LD5-1} and~\ref{LD5-2}. 
Since $\gamma_L(C_{15})=6$, in $(6,4,2,1,1,1)$ the 6-set $A_1$ must form LD-coalitions with the three singletons. 
In $(6,5,1,1,1,1)$, these lemmas imply $A_2$ (the 5-set) forms an LD-coalition only with one singleton, say 
$A_3$; thus $A_1$ pairs exactly with $A_4, A_5, A_6$. By Lemma~\ref{LD5-2}, $\{A_4,A_5,A_6\}$ induces a path in
$C_{15}$. But then $A_2\cup A_3$ fails to dominate the central vertex in the induced path by $A_4\cup A_5\cup A_6$, 
a contradiction.

Finally, let $n = 12,14$ and  $n\ge 16$. By Lemma \ref{TCL6}, for any cycle we have $C_L(C_n)\leq 6$. 
Now, we find the locating coalition partitions of size 6 for $C_n$. These locating coalition partitions 
are shown in {Figure~\ref{FigLCnEvOd}}. The left cycle has even order $n =12,14,16,\ldots$, where   
$k,m \ge 1$. The right cycle has odd order $n=17,19,21\ldots$, for  $m,k,r \ge 1$. This completes the proof.

\begin{figure}[!ht]
\centering
\includegraphics[width=0.75\linewidth]{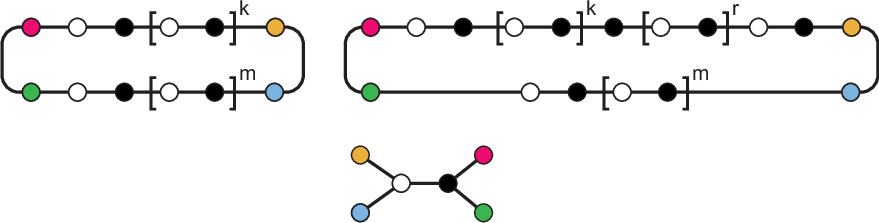}
\caption{LDC-partitions of cycles with $C_L(C_n)=6$ and the corresponding locating coalition graph.}  
\label{FigLCnEvOd}
\end{figure}

\end{proof}

The proof of Lemma~\ref{TCL6} can be easily adapted to paths. The LDC-partition of 
size $5$ shown in {Figure~\ref{FigCL5}} is suitable for every path $P_n$ of order $n\ge 7$.

\begin{lemma}\label{TLP6}
For the path  $P_n$ of order $n\geq 3$, $C_{L}(P_n) \le 6$.
\end{lemma}

\begin{proposition} For any path $P_n$,   
$$
C_L(P_n) = \left\{
\begin{array}{rl}
3,  &  \mbox{if } n = 3; \\
4,  &  \mbox{if } n = 4,5,6; \\
5,  &  \mbox{if } 7 \le n \le 15; \\
6,   & \mbox{if } n \ge 16. \\
\end{array}
\right.
$$
\end{proposition}

\begin{proof} Let $P_n:v_1,v_2,\ldots,v_n$ be a path. If $3\leq n \leq 4$, then it is easy to check that 
$C_L(P_n)=n$.  For $n=5$, we may deduce that $C_{L}(P_{5})<5$, because the
central vertex $v_3$ of the path $P_5$ cannot form an LD-coalition with any other vertex of the path.
Besides, since the partition $\{\{v_1,v_5\},\{v_2\},\{v_3\},\{v_4\}\}$ is an LDC-partition, we conclude 
that $C_L(P_5)=4.$ 

Next, consider $n=6.$ Since $\gamma_L(P_6)=3,$ we have that $C_L(P_6)\leq 5,$ because no pair of singleton
sets can form an LD-coalition.
Let $\Pi$ be an LDC-partition of cardinality $5$. Then $\Pi$ consists of a set of cardinality $2$ 
and four singleton sets. It must be noted that every singleton set necessarily forms a locating coalition with 
a set of cardinality $2$, which is impossible, as $P_6$ cannot contain four LD-sets of size $3$ sharing two
common vertices. Hence, $\Pi$ does not exist and therefore  $C_L(P_6)\leq 4$. The partition $\{V_1=\{{v_2,v_4\}},$
$V_2=\{v_1,v_3\},$ $ V_3=\{v_5\},V_4=\{v_6\} \}$ is an LDC-partition of $P_6$. The set $V_1$ forms an LD-coalition
with each of the sets $V_3$ and $V_4$, and the set $V_2$ forms an LD-coalition with each of the sets $V_3$ 
and $V_4$.  So, the equality $C_L(P_6)=4$ holds.

We show next that $C_L(P_n)\neq 6$, where $7\le n \le 15$. Suppose, to the contrary, that $C_L(P_n)=6$ for 
these paths. Using an idea similar to that in the proof of Proposition \ref{p1}, one can construct a table 
describing the different possibilities, see Table~\ref{TabPart1}. Similarly to Table \ref{TabPart}, the same result 
holds where $7\le n \le 11$ and $n =13,15$. Consequently, it suffices to consider the table for $ n=12$ and 
$n=14$. Note that the labels $1$ and $2$ are assigned to Lemma \ref{t} and Theorem \ref{t1}, respectively. By 
Lemma \ref{t} and Theorem \ref{t1}, we derive that almost all LDC-partitions of cardinality $6$ cannot exist.

\begin{figure}[t]
\centering
\includegraphics[width=.85\linewidth]{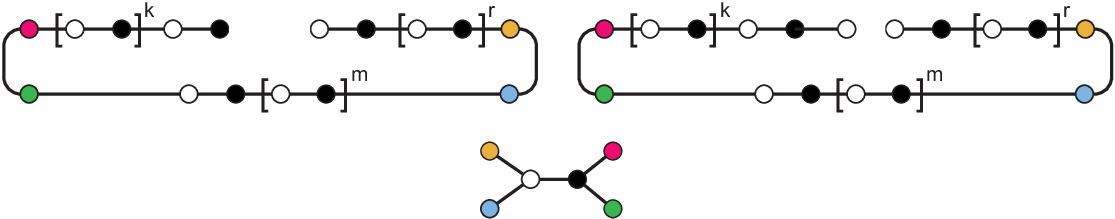}
\caption{LDC-partitions of paths with $C_L(P_n)=6$ and the corresponding locating coalition graph.}  
\label{FigLPnEvOd}
\end{figure}

Next, we consider the remaining partitions $(4,4,1,1,1,1)$ and $(4,3,2,1,1,1)$ for $n=12$; and
the partitions $(5,5,1,1,1,1)$, $(5,4,2,1,1,1)$, and $(5,3,3,1,1,1)$ of $P_{14}.$

Let us start with the case $n=12.$ Since $\gamma_L(P_{12})=5$, all singleton sets must form an LD-coalition
with a $4$-set. Analogously to previous results, let $[g_1,g_2,\ldots,g_6]$ be the
gaps configuration of a $5$-element $A'$ set in $V(P_{12}).$ First observe that, if $A'$ is an LD-set
then:
\begin{itemize}
    \item There are no $3$-gaps, since otherwise $A'$ would fail to dominate the path.
    \item No two $2$-gaps are consecutive, as the vertex separating them would have neighbors with 
    identical neighborhoods in $A'$, contradicting that $A'$ is an LD-set.    
    \item A gap configuration cannot begin or end with a $2$-gap; otherwise, $A'$ would not dominate 
    the whole set $\{v_1, v_{12}\}$.    
    \item No gap configuration can start with $[1,2,\ldots]$ or end with $[\ldots,2,1]$, since in either 
    case two vertices would share the same neighborhood in $A'$, a contradiction.    
\end{itemize}

For a $5$-set $A'$, let $[g_1,\ldots,g_6]$ be its gap configuration, and let $p$, $q$, and $r$ denote 
the numbers of $0$-, $1$-, and $2$-gaps, respectively. Then $p+q+r=6$ and $q+2r=7$, 
hence $(p,q,r)\in\{(0,5,1),(1,3,2)\}$.

\begin{table}[!b] 
\small
\centering
\caption{All vertex partitions of path $P_n$.}  \label{TabPart1}
\resizebox{\textwidth}{!}{%
\begin{tabular}{c|c| lllll}  \hline
 $n$  & $\gamma_{L}(P_n)$   &                                  &    & partitions   &   &             \\  \hline

12     &        5                             &  (7,1,1,1,1,1)$^{1,2}$     &  (6,2,1,1,1,1)$^{1,2}$    & (5,3,1,1,1,1)$^{1,2}$  &   (5,2,2,1,1,1)$^{1,2}$ &   (4,4,1,1,1,1)  \\ 
          &                                     &  (4,3,2,1,1,1)    &  (4,2,2,2,1,1)$^{1,2}$   &  (3,3,3,1,1,1)$^1$  &  (3,3,2,2,1,1)$^1$  &   (3,2,2,2,2,1)$^1$ \\ 
          &                                     &  (2,2,2,2,2,2)$^1$ &  &  \\   \hline
          
14     &       6                             &   
(9,1,1,1,1,1)$^{1,2}$      &  (8,2,1,1,1,1)$^{1,2}$     & (7,3,1,1,1,1)$^{1,2}$   &   (6,4,1,1,1,1)$^{1,2}$   &   (5,5,1,1,1,1)   \\ 
          &                                     &  
          (7,2,2,1,1,1)$^{1,2}$      &  (6,3,2,1,1,1)$^{1,2}$     & (5,4,2,1,1,1)   &   (5,3,3,1,1,1) & (4,4,3,1,1,1)$^1$       \\ 
          &                                     &  
          (6,2,2,2,1,1)$^{1,2}$      &  (5,3,2,2,1,1)$^{1,2}$     & (4,4,2,2,1,1)$^1$   &   (4,3,3,2,1,1)$^1$ & (3,3,3,3,1,1)$^1$       \\ 
 &                                     &  
          (4,2,2,2,2,1)$^1$      &  (4,3,2,2,2,1)$^1$     & (3,3,3,2,2,1)$^{1,2}$   &   (4,2,2,2,2,2)$^{1,2}$ & (3,3,2,2,2,2)$^1$       \\ \hline
          
\end{tabular}
}
\end{table}

If $(p,q,r)=(0,5,1)$, the only possible configurations are $[1,1,2,\allowbreak 1,1,1]$ and 
$[1,1,1,\allowbreak 2,1,1]$, yielding $\{v_2,v_4,\allowbreak v_7,\allowbreak v_9,v_{11}\}$ and 
$\{v_2,v_4,\allowbreak v_6,\allowbreak v_9,v_{11}\}$. If $(p,q,r)=(1,3,2)$, the configurations 
are $[0,2,1,\allowbreak 2,1,1]$ and $[1,1,2,\allowbreak 1,2,0]$, yielding 
$\{v_1,v_4,\allowbreak v_6,\allowbreak v_9,v_{11}\}$ and $\{v_2,v_4,\allowbreak v_7,\allowbreak v_9,v_{12}\}$. 

Since these four sets do not share four common vertices, $(4,3,2,1,1,1)$ is infeasible; moreover, 
$(4,4,1,1,1,1)$ is impossible, as no two disjoint $4$-sets occur among them.

We now consider the case $n=14$. The only configurations to examine are 
$(5,5,1,\allowbreak 1,1,1)$, $(5,4,2,\allowbreak 1,1,1)$, and $(5,3,3,\allowbreak 1,1,1)$ in $P_{14}$.

Since $\gamma_L(P_{14})=6$, every singleton must form an LD-coalition with some $5$-set. 
Let $A=\{w_j : j\in[5]\}$ be a $5$-set which is not an LD-set but forms an LD-coalition with the singleton 
$\{w_6\}$. If $N[A]$ does not induce a connected subpath of $P_{14}$, then $A\cup\{v\}$ is not a dominating 
set for any singleton $v$, and hence no configuration is feasible.

Let us, otherwise, suppose now that $|N[A]|\le 11$. Then at most one singleton can extend $A$ to 
a dominating set. Consequently, none of the three configurations is possible.

Assume therefore that $12\le |N[A]|\le 13$ and, without loss of generality, that 
$V\setminus N[A]\subseteq \{v_{13},v_{14}\}$. In this case, at most the singletons $S_1=\{v_{13}\}$ and 
$S_2=\{v_{14}\}$ can form an LD-set together with $A$. Hence the configurations $(5,4,2,\allowbreak 1,1,1)$ and 
$(5,3,3,\allowbreak 1,1,1)$ are not feasible.

Moreover, the vertex $v_{14}$ can only be dominated by itself or by $v_{13}$. Therefore, no additional 
$5$-set can form an LD-set together with two singletons distinct from $S_1$ and $S_2$. This excludes the
configuration $(5,5,1,\allowbreak 1,1,1)$.

Lastly, assume that $A$ is a dominating set in $P_{14}$. 
Reasoning as in the previous cases, let $[g_1,\ldots,g_6]$ be the gap configuration of $A$ in $P_{14}$. 
Since $A$ is dominating, we have $g_j\in\{0,1,2\}$ for every $j$. 
Let $p,q,r$ denote the number of $0$-, $1$-, and $2$-gaps, respectively. 
It follows that $p+q+r=6$ and $q+2r=9$. 

Since $A$ is not an LD-set but $A\cup\{w_6\}$ is, no three consecutive $2$-gaps may occur, 
and hence $(p,q,r)=(0,3,3)$. 
Moreover, $g_1\neq 2$ and $g_6\neq 2$, since otherwise $v_1$ or $v_{14}$ would not be dominated. 
Thus, the only possible gap configurations are $[1,2,1,2,2,1]$ and $[1,2,2,1,2,1]$.

In either case, the pairs $\{v_1,v_3\}$ and $\{v_{12},v_{14}\}$ are not located-dominated by $A$, 
and therefore $A\cup S$ is not an LD-set for any singleton $S$.

Finally, let $n \ge 16$. By Lemma \ref{TLP6}, for any path we have $C_L(P_n)\leq 6$. 
Now, we find the locating coalition partitions of size 6 for $P_n$. These locating coalition partitions are 
shown in Figure~\ref{FigLPnEvOd}. The left path has even order $n=16, 18, 20, \dots $, where $m,k,r \ge 1$.
The right path has odd order $n=17, 19, 21, \dots $, for $m,k,r \ge 1$. 
This completes the proof.
\end{proof}

\section{Graphs with $C_L(G) \in \{n-1,n\}$} \label{n-1,n}

In this section, we study connected graphs of order $n\geq3$ such that $C_L(G)$ belongs to $\{n-1,n\}$. 
We first provide a characterization of graphs $G$ of order
$n\geq 3$ having $C_{L}(G)=n.$ Let $H$ denote the graph of order $4$ obtained
from the cycle $C_{3}$ by attaching a new vertex to one vertex
of $C_{3}$.

\begin{proposition}
\label{prop1} Let $G$ be a graph of order $n\geq3.$ Then $C_{L}(G)=n$
if and only if $G\in\{P_{3},C_{3},P_{4},C_{4},K_{4}-e,H,C_{5},C_{5}+e\}$.
\end{proposition}

\textbf{Proof. }If
$G\in\{P_{3},C_{3},P_{4},C_{4},K_{4}-e,H,C_{5},C_{5}+e\}$, then one
can check that $C_{L}(G)=n$, since for every vertex of $G,$ there is
another vertex which together form a locating-dominating set in $G.$
To prove the converse, assume that $C_{L}(G)=n.$ By Corollary
\ref{C_L=n}, $n\in\{3,4,5\}.$ Clearly, if $n=3$, then
$G\in\{P_{3},C_{3}\},$ while if $n=4,$ then
$G\in\{P_{4},C_{4},K_{4}-e,H\}$.  We note that the only 
graph of order $4$ that has been excluded is the graph $K_{4},$ since
$\gamma_{L}(K_{4})=3.$ In the following we can assume that $n=5.$ It
is worth noting that Harary has provided in his book \cite{harary} all
graphs of order at most 6. Taking into account the fact that $G$ is
connected of order $5$, there are $21$ graphs to be considered (see
pages $216-217$). Among these graphs, there are exactly $10$ with a
 {locating-domination number} equal to $2$, and of these $10$ graphs, we
only keep two graphs, namely $C_{5}$ and $C_{5}+e.$ The $8$ other
graphs are excluded because each contains a vertex that cannot form
with another vertex a locating-dominating set. $\Box$\medskip

\begin{figure}[ht]
	\begin{center}
		\includegraphics[width=0.4\linewidth]{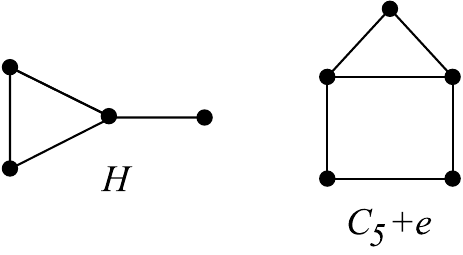}
		\caption{ The graphs $H$, $C_{5}+e$
		}
		\label{chessboard2}
	\end{center}
\end{figure}

Next we characterize all trees $T$ of order $n\geq3$
satisfying $C_{L}(T)=n-1.$ For this aim, we first recall the lower bound on the
 {locating-domination number} due to\ Slater \cite{slater87} established for any
nontrivial tree $T,$ where he showed that $\gamma_{L}(T)\geq\left\lceil
(n+1)/3\right\rceil$.
\begin{proposition}
Let  {$T$} be a tree of order $n\geq3.$ Then $C_{L}(T)=n-1$ if and only if
$T\in\{P_{5},K_{1,3}\}.$
\end{proposition}

\textbf{Proof. }Assume that $C_{L}(T)=n-1.$ By Proposition \ref{prop0},
$\gamma_{L}(T)\leq3.$ If $\gamma_{L}(T)=2,$ then by Slater's lower bound,
$n\leq5,$ and clearly $T\in\{P_{3},P_{4},P_{5}\}.$ Moreover, by Proposition
\ref{prop1}, $T\notin\{P_{3},P_{4}\},$ and thus we deduce that $T=P_{5}.$ In
the following, we can assume that $\gamma_{L}(T)=3.$ Again, Slater's lower
bound implies that $T$ has order $n\leq8.$ Now, if $T$ contains four leaves or
more, then $\gamma_{L}(T)\geq4,$ leading to a contradiction. Therefore $T$
contains two or three leaves. Since a tree with two leaves is a path, we then
have $T\in\{P_{6},P_{7},P_{8}\}.$ The path $P_{8}$ is excluded since
$\gamma_{L}(P_{8})=4>3,$ and the two other remaining paths are also excluded
since $C_{L}(T)<n-1.$ Hence we can assume that $T$ has three leaves. Clearly,
$\Delta(T)=3$ and $T$ has exactly one vertex of degree 3, say $x$. Therefore,
any other vertex of $T$ has degree one or two, and in such a case, $T$ is a
tree obtained from a star $K_{1,3}$ by subdividing each edge of the star
several times or not at all. As a result, $n\in\{4,5,6,7,8\},$ and thus we
look at each case separately. 

If $n=4,$ then $T=K_{1,3},$ and so a tree given in the statement. If $n=5,$
then $T=T_{1}$ and clearly $C_{L}(T_{1})<n-1=4.$ If $n=6,$ then $T\in
\{T_{2},T_{3}\},$ and again $C_{L}(T)<n-1=5.$ If $n=7,$ then $T\in
\{T_{4},T_{5},T_{6}\},$ where $\gamma_{L}(T_{4})=4>3,$ and $C_{L}(T)<n-1=6$
for any $T\in\{T_{5},T_{6}\}.$ Finally, if $n=8,$ then $T\in\{T_{7}%
,T_{8},T_{9},T_{10}\},$ where it can be easily seen that $\gamma_{L}(T)=4>3$
for any $T\in\{T_{7},T_{8},T_{9}\}$ and $C_{L}(T_{10})<n-1=7.$ The converse 
is easy to show, which completes the proof. $\Box$\newline

\begin{figure}[ht]
	\begin{center}
		\includegraphics[width=0.9\linewidth]{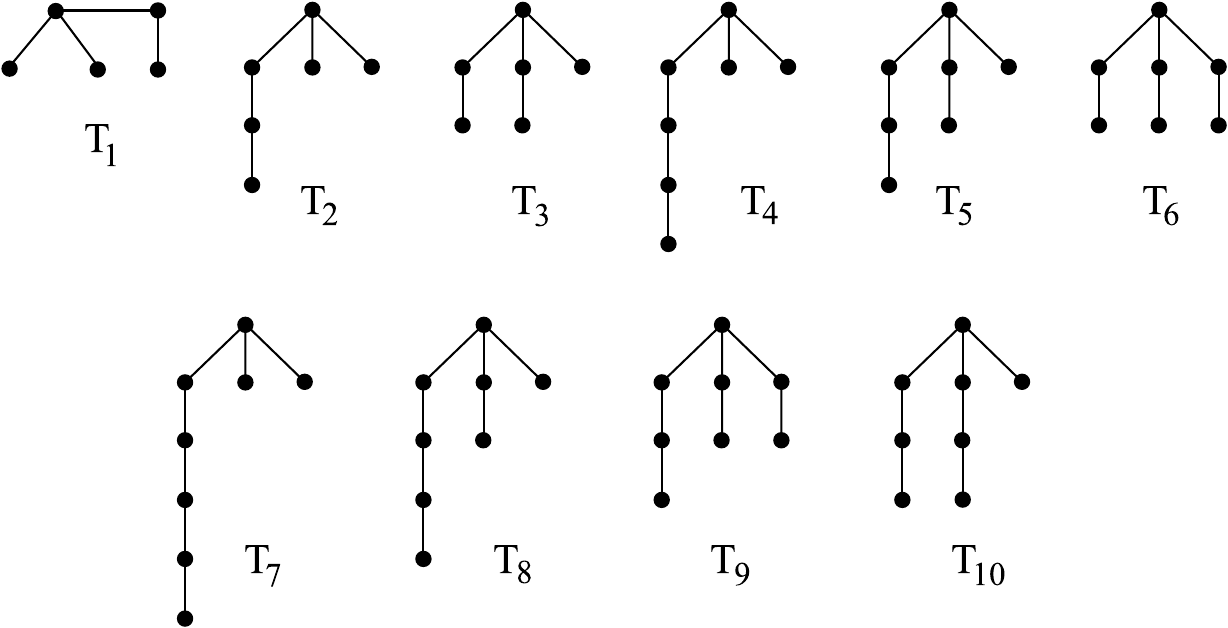}
		\caption{ The trees $T_{i}$, $i\in[10]$
		}
		\label{chessboard}
	\end{center}
\end{figure}

\section{Computational results} \label{computational}

We now study the computational complexity of determining the  {locating-domination coalition 
number} of a graph, formally defined as the following decision problem.

\decisionpb{ {LOCATING-DOMINATING COALITION} (LDC for short)}
{A graph $G$, an integer $k$.}{Do we have $C_L(G)\geq k$?}{0.7}

The LDC-problem is a complex vertex partition with 
both positive and negative constraints. In particular, it requires that the 
union of at least one pair of subsets be an LD-set, and no single subset be an 
LD-set. These complex conditions within the framework of a partitioning problem 
maintain, or naturally increase, the computational difficulty because the 
fundamental subproblem of figuring out the minimum size of an LD-set is 
NP-hard itself.

However, we can prove that this problem is tractable when restricted to
some special classes of graphs. To do that, we will make use of the 
results of Courcelle et al. \cite{courcelle}.

First, we recall what a $t$-representation of a graph $G$ is. Let $G=(V, E)$ be a 
graph, and let $[t]$ be a set of labels. A $t$-expression of the graph $G$ is a 
construction of the graph by means of the following four operations:
\begin{enumerate}
    \item $\eta_i(v):$ Creates a graph consisting of a single vertex $v$ labeled 
    with $i \in [t]$.
    \item $G_1 \oplus G_2:  $ Forms the disjoint union of two labeled graphs $G_1$ 
    and $G_2$, adding no edges between them.
    \item $\rho_{i,j}(G) \quad \text{with } i \neq j.$
    Add edges between all vertices in $G$ labeled with $i$ and all vertices labeled 
    with $j$.
    \item $\chi_{i \to j}(G) \quad \text{with } i \neq j.$
    Change the label $i$ to the label $j$ for all vertices in $G$ that were labeled $i$.
\end{enumerate}
The clique-width of a graph $G$, denoted by $\text{cw}(G)$, is the minimum integer $t$ such 
that $G$ admits a $t$-expression.

The following theorem is very powerful and allows many problems to
  be solved efficiently on graphs of bounded clique-width. Note that
  the theorem applies in particular to graphs of tree-width $t$, for
  which one can obtain in linear time a tree-decomposition of width at
  most $t' = 2t + 1$~\cite{K21} and from that, a $\left(3\cdot
  2^{t'-1}\right)$-expression~\cite{CR05}. 

\begin{theorem}{\emph{(Courcelle et al. \cite{courcelle})}} 
Let $\mathcal{C}$ be a class of graphs with bounded clique-width, 
$\text{cw}(G) \le t,$ for some fixed constant $t$. 
Let $\Phi$ be a fixed formula of Monadic Second-Order Logic (MSOL$_1$). 
If a $t$-expression of the input graph $G$ is given as part of the 
input, then the decision problem whether $G$ satisfies $\Phi$ can be solved in linear time.
\end{theorem}

We aim to express the decision problem associated with the LDC-partition
problem in the framework of MSOL$_1$, on the logical structure 
$\langle \{G, \mathcal{R}\} \rangle$, where $\mathcal{R}$ is an equivalence 
relation defined as follows: $\mathcal{R}(u,v)$ holds if and only if 
$uv \in E(G)$.

\begin{theorem}\label{thm:CW}
The LDC-problem can be solved in linear $O(f(k,t))$ time when restricted to graphs 
with clique-width at most $t$, if a $t$-expression is provided as part of 
the input.

\end{theorem}

\begin{proof}
Let $G=(V,E)$ be a graph and let $k$ be a positive integer. We let the formula 
$\Phi$ be defined as
\[ \Phi : \exists \{V_1,\ldots,V_k\} \, \big( \Phi_P 
\wedge \Phi_{NLD}\wedge \Phi_{LDC} \big),
\]
where $\Phi_P$, $\Phi_{NLD}$ and $\Phi_{LDC}$ are subformulas that guarantee the 
sets $V_j$ constitute a locating-coalition partition. 

The first subformula must guarantee that the subsets $V_j$ form a partition of
$V$. We can express it as
\[ \Phi_P \equiv \left( \bigwedge_{i=1}^k \bigwedge_{j=i+1}^k \neg \exists v 
\left(  V_i(v) \land  V_j(v) \right) \right) \land \left( \forall v 
\bigvee_{i=1}^k \left(  V_i(v) \right) \right)  \]
where the first parenthesized expression ensures that the sets are pairwise disjoint 
and the second ensures that they cover the whole set of vertices.

The second subformula, $\Phi_{NLD}$, will permit us to assure that no subset $V_j$ is
an LD-set. For the sake of simplicity, we define the following three predicate formulas: 
\[ \begin{aligned}
    (i) \ N[u,v] &\equiv \Big( (u=v) \lor \mathcal{R}(u,v) \Big)
    \\[1em]
    (ii) \ Dom(V_j) &\equiv \forall v \bigg( 
     \neg V_j(v) \implies \exists u \Big(  V_j(u) \land \mathcal{R}(v,u) \Big) \bigg)
    \\[1em]
    (iii) \ Loc(V_j) &\equiv \forall u, v \Bigg(  
     \neg V_j(u) \land  \neg V_j(v) \land u \neq v 
    \implies 
    \\
    &\implies \exists z \bigg(  
    \Big(  ( N[u,z] \land  V_j(z)) \oplus (N[v,z] \land V_j(z)) \Big) \bigg) \Bigg)  
\end{aligned}
\]
where the operator $\oplus$ denotes the logical exclusive OR, i.e., $A \oplus B \equiv 
(A \land \neg B) \lor (\neg A \land B).$
The predicate $N[u,v]$ is true if and only if the vertex $u$ belongs to $N[v]$. 
Additionally, $Dom(V_j)$ holds if the subset $V_j$ is a dominating set, 
and $Loc(V_j)$ holds if $V_j$ is a locating set.

\medskip 

Combining all of the above, we may express the formula, $\Phi_{NLD}$, as follows,
\[ \Phi_{NLD} \equiv \bigwedge_{j=1}^k \neg \bigg( 
    Dom(V_j) \land Loc(V_j) \bigg).   \]
which is true if and only if no subset $V_j$ is an LD-set.

The third subformula, $\Phi_{LDC},$ must guarantee that for each $V_j$ there is a subset
$V_i$ such that $V_j\cup V_i$ is an LD-set.
\[ \Phi_{LDC} \equiv \bigvee_{i=1}^k \bigvee_{j=i+1}^k Dom(V_j \cup V_i) \land Loc(V_j\cup V_i)\]

Thus, finally, we have expressed the LDC-problem by means of an MSOL$_1$ formula. 
By applying the result of Courcelle, the theorem holds.
\end{proof}

\section{Open problems}

We close the work with the following questions.

\begin{enumerate}

        \item Establish upper and lower bounds on $C_L(G)$ in terms of minimum and maximum degree.

        \item Does every graph of diameter 2 satisfy $C_L(G)\geq2$? (Is it true that $K_1$ and $K_2$ are the only graphs without a locating-coalition?)

          \item Is {Proposition~\ref{prop0} tight}?
      
        \item Is the LDC-problem (to decide whether $C_L(G)\geq k$) NP-hard? This is especially interesting for constant $k$, since Theorem~\ref{thm:CW} applies for constant $k$. Note that the LDC-problem is in NP, since a locating-coalition partition is a witness that can be checked in polynomial time.
       
\end{enumerate}

\section{Acknowledgements}

The work of A.A. Dobrynin and H. Golmohammadi was supported by the state contract of the 
Sobolev Institute of Mathematics (project number FWNF-2026-0011). Florent Foucaud was supported 
by the French government IDEX-ISITE initiative 16-IDEX-0001 (CAP 20-25), the International Research
Center ``Innovation Transportation and Production Systems'' of the I-SITE CAP 20-25, and the ANR
project GRALMECO (ANR-21-CE48-0004). J. C. Valenzuela-Tripodoro was partially funded by the 
Spanish Ministry of Science, Innovation and Universities through project PID2022-139543OB-C41 and 
also by the European Commission's Horizon Europe Research and Innovation programme through the 
Marie Sklodowska-Curie Actions Staff Exchanges (MSCA-SE) under Grant Agreement no. 101182819 
(COVER: (C)ombinatorial (O)ptimization for (V)ersatile Applications to (E)merging u(R)ban Problems).

\end{document}